\documentclass{amsart}

\usepackage{amsmath,amsfonts,graphicx}
\usepackage{commath}
\usepackage{mathtools}
\usepackage{amssymb}
\usepackage[citestyle=alphabetic, bibstyle=alphabetic, sorting=ynt]{biblatex}

\usepackage{todonotes}
\usepackage{hyperref}
\usepackage[capitalise]{cleveref}

\usepackage{tikz}
\usetikzlibrary{decorations.pathreplacing,arrows.meta}
\tikzset{>={Stealth[scale=.4]}}
\usetikzlibrary{calc}

\theoremstyle{plain}
\newtheorem{theorem}{Theorem}[section] 

\theoremstyle{definition}
\newtheorem{definition}[theorem]{Definition} 

\newtheorem{proposition}[theorem]{Proposition}
\newtheorem{lemma}[theorem]{Lemma}

\newtheorem{maintheorem}{Theorem}

\newcommand{\R}{\mathbb{R}}

\newcommand{\N}{\mathbb{N}}

\begin{document}
\title{Unique equilibrium states for Viana maps with small potentials}

\author{Kecheng Li}
\address{Department of Mathematics, Tufts University, 
Medford, MA 02155}
\email{kecheng.li@tufts.edu}
\urladdr{} 

\begin{abstract}
    We investigate the thermodynamic formalism for Viana maps—skew products obtained by coupling an expanding circle map with a slightly perturbed quadratic family on the fibers. For every Hölder potential \(\varphi\) whose oscillation is below an explicit threshold, we show that an equilibrium state not only exists but is unique and satisfies an upper level-2 large-deviation principle. All of these conclusions persist under sufficiently small perturbations of the reference map.
\end{abstract}

\maketitle

\section{Introduction}
In this paper, we establish the existence and uniqueness of equilibrium states for Viana maps under a natural small-oscillation condition on the potential. The equilibrium state we derive is a Gibbs measure that satisfies a global upper Gibbs property. This enables us to obtain an upper level-2 large deviation principle.

Viana maps form an influential class of non-uniformly expanding surface endomorphisms obtained as skew–products of an expanding circle map with a perturbed quadratic family on the fibers \cite{MR1471866}.  Along the one-dimensional \emph{central} bundle the derivative is, in principle, expanding; however, whenever an orbit slips into the critical line $t=0$ the quadratic term folds the fiber, sharply reducing the local derivative and diluting the accumulated expansion.  This intermittent weakening prevents the system from being Anosov while still permitting a rich thermodynamic theory.

Let \(f :M \rightarrow M\) be a continuous map on a compact smooth manifold. Among the invariant
probability measures for the system, thermodynamic formalism identifies distinguished measures called equilibrium states; these are measures that maximize the quantity
\(h_{\mu}(f)+ \int \varphi \,d\mu\), where \(\varphi:M \rightarrow \R\) is a potential function.

Classically, Bowen showed that for a \emph{mixing uniformly hyperbolic} system, the combination of expansivity and the specification property ensures both existence and uniqueness of equilibrium states for every Hölder potential \cite{Bowen1974SomeSW}.  Climenhaga and Thompson later extracted the main ingredient of Bowen’s argument and recast it in a flexible \emph{orbit-decomposition} framework that works beyond uniform hyperbolicity \cite{MR3552538}.  Their technique measures the topological pressure carried by ``obstructions'' to expansivity and specification; if these obstructions have strictly smaller pressure than the whole system, a unique equilibrium state still exists. In recent years, the Climenhaga–Thompson machinery has been applied extensively to construct equilibrium states in a range of settings:
\begin{enumerate}
\item Mañé diffeomorphisms \cite{CFT} and the Katok map~\cite{WANG21}.
\item geodesic flows on manifolds of non‑positive curvature \cite{BCFT,CT_22} and on surfaces without conjugate points \cite{CKW};
\item Lorenz flows with singularities \cite{PYY};
\item geodesic flows on flat surfaces with singularities \cite{CCESW};
\end{enumerate}

Using this generalization, we show that for Viana maps every Hölder continuous function whose oscillation is below an explicit threshold admits a \emph{unique} equilibrium state. The proof adapts the Climenhaga–Thompson machinery: we carve the dynamics into a ``good'' core that enjoys specification and the Bowen property and a ``bad'' tail whose pressure we show is strictly lower than the topological pressure.

One benefit of \cite{Bowen1974SomeSW} and \cite{MR3552538} is the construction of the unique equilibrium state as a Gibbs measure. In \cite{Bowen1974SomeSW}, the lower Gibbs property is essential in ruling out mutually singular equilibrium states. This approach is generalized in \cite{MR3552538}, in which Climenhaga and Thompson derive the lower Gibbs property of equilibrium state for the ``essential collection of orbit segments'' which dominates in pressure as well as a global upper Gibbs property for all orbit segments. 

Based on this property, we are able to deduce an upper bound for the level-2 large deviation principle for the equilibrium state of Viana maps. In general, the large deviation principle describes the exponential rate of convergence of time average to the space average with respect to a given measure.

Before stating the main theorem, recall Viana’s construction.  
\begin{definition}[Viana map]
    Consider the map  
    \[
    f_{\alpha}(\theta,t)=\bigl(d\theta \mod 1,\;a_0-t^2+\alpha \sin (2\pi\theta)\bigr),
    \qquad
    (\theta,t)\in S^{1}\times \R,
    \]
Here \(d \geq 2\) is an integer, \(\alpha\) is a small real number and the parameter \(a_0 \in (1, 2)\) is such that the map \(h(t) = a_0 - t^2\) has a pre-periodic (but not periodic) critical point. For \(\alpha\) small enough there is an interval \(I \subset (-2,2)\) for which \(f_{\alpha}(S^1\times I) \subset \mathrm{int}(S^1\times I)\).
\end{definition}

In \cite{MR1471866}, Viana constructed this quadratic skew-products over the linear strongly expanding map of the circle and proved that if \(d\geq 16\), then there exists a positive constant $c_0 > 0$ such that any map $f$ sufficiently $C^{3}$-close to $f_{\alpha}$ has both Lyapounov exponents greater than $c_0$ Lebesgue-a.e.. 

Buzzi, Sester and Tsujii later in \cite{MR2018605} proved that if \(d\ge2\), then for every sufficiently small \(\alpha>0\) the same is true in a $C^{\infty}$-neighborhood.

The principal result of this paper is the following:

\begin{maintheorem}\label{main}
    For all \(d\ge16\) and \(\epsilon>0\) sufficiently small, there is an \(\alpha_0>0\) such that for all \(\alpha\le\alpha_0\) and all Hölder potentials  with \(\sup \varphi-\inf \varphi <\frac{c_0}{2}-\log \frac{d+\epsilon}{d-\epsilon}\), the Viana map \(f_{\alpha}\) has a unique equilibrium state \(\mu\) for \(\varphi\). Moreover, \(\mu\) has the upper level-2 large deviations principle.
\end{maintheorem}

Our proof is designed so that every step survives small perturbations to the dynamics. In particular, the expansivity, specification, and Bowen property we establish for \(f_{\alpha}\) remain valid for all maps in a sufficiently small \(C^{3}\) neighborhood.  This robustness yields the following extension:

\begin{maintheorem}\label{main_B}
    The same result in \cref{main} is true for any sufficiently small \(C^{3}\) perturbation of \(f_{\alpha}\) (which is not necessarily a skew-product).
\end{maintheorem}

Henceforth we use the term \emph{Viana map} for any map \(f\) that lies in a sufficiently small \(C^{3}\)-neighborhood of the reference skew product \(f_{\alpha}\). Unless explicitly stated otherwise, every result we state for a Viana map therefore applies simultaneously to \(f_{\alpha}\) itself and to all its \(C^{3}\) perturbations within this neighborhood.

We briefly situate our results within the existing literature on Viana maps. Alves proved existence and uniqueness of the SRB measure \cite{MR1743717}.  
Arbieto, Matheus and Senti later established existence of equilibrium states for Hölder potentials with sufficiently small oscillation in an unpublished preprint \cite{arbieto2005maximalentropymeasuresviana}.  
Pinheiro and Varandas extended uniqueness to every Hölder potential whose oscillation is below a threshold \cite{pinheiro2023thermodynamicformalismexpandingmeasures}. Recently, using the countable Markov coding approach, Lima, Obata and Poletti also obtained uniqueness of the measure of maximal entropy \cite{lima2024measuresmaximalentropynonuniformly}

The present work follows a different route: employing the decomposition–obstruction scheme of Climenhaga–Thompson, we show that the unique equilibrium state guaranteed by \cref{ct} is in fact a Gibbs measure. An immediate consequence of this Gibbs characterization yields the large-deviation principle.
\\

\paragraph*{We conclude the introduction with a roadmap of the paper.}
\begin{itemize}
    \item \textbf{\cref{sec:background}} reviews the decomposition–obstruction framework of Climenhaga--Thompson and the necessary concepts from thermodynamic formalism.
    
    \item \textbf{\cref{sec:viana-maps}} recalls Viana’s construction and compiles the dynamical properties used later.
    
    \item \textbf{\cref{sec:hyperbolic-times}} discusses non-uniformly expanding maps and summarises Alves’s notion of \emph{hyperbolic times}.
    
    \item \textbf{\cref{sec:exp}} proves that Viana maps are non-uniformly expansive by showing that the set of non-expansive points carries pressure strictly below the topological pressure.
    
    \item \textbf{\cref{sec:spec}} establishes a non-uniform specification property by demonstrating that most orbit segments satisfy specification.
    
    \item \textbf{\cref{sec:proof-main}} combines these ingredients to prove the main theorem and, as a consequence of the Gibbs property, derives a large-deviation principle.
\end{itemize}

\medskip
\noindent\textbf{Acknowledgments.}  
The author thanks Boris Hasselblatt for invaluable guidance and many helpful discussions.

\section{Background and main technique}\label{sec:background}
We state the preliminary definitions needed for the technique and introduce how to apply the technique to deduce the desired thermodynamic formalism.
\subsection{Pressure}
Let \(f : X \to X\) be a continuous map on a compact metric space. We identify the set \(X \times \mathbb{N}\) with the space of finite orbit segments by identifying \((x, n)\) with \(\bigl(x, f(x), \ldots, f^{n-1}(x)\bigr).\)

The \emph{Bowen metric} of order \(n\) associated to \(f\) is defined by
\[
d_n(x,y) \;=\; \max_{0 \le k < n} d\bigl(f^k(x), f^k(y)\bigr).
\]
Given \(x \in X\), \(\epsilon > 0\), and \(n \in \mathbb{N}\), the \emph{Bowen ball} of order \(n\) with center \(x\) and radius \(\epsilon\) is
\[
B_n(x, \epsilon) 
\;=\; 
\Bigl\{\,y \in X : d_n(x, y) < \epsilon \Bigr\}.
\]
A set \(E \subset X\) is called \((n, \epsilon)\)-separated if
\[
d_n(x, y) \;\ge\; \epsilon 
\quad\text{for all } x, y \in E \text{ with } x \neq y.
\]

Given \(\mathcal{D} \subset X \times \mathbb{N}\), we interpret \(\mathcal{D}\) as a collection of finite orbit segments. Define
\[
\mathcal{D}_n \;=\; \bigl\{\, x \in X : (x, n) \in \mathcal{D} \bigr\},
\]
the set of initial points of all length-\(n\) orbit segments in \(\mathcal{D}\). We then consider the corresponding partition sum
\[
\Lambda_n^{\mathrm{sep}} \bigl(\mathcal{D}, \varphi, \epsilon; f\bigr)
\;=\;
\sup \Bigl\{ \sum_{x \in E} e^{S_n \varphi(x)} 
\;:\;
E \subset \mathcal{D}_n \text{ is } (n,\epsilon)\text{-separated} 
\Bigr\}.
\]

\begin{definition}[Pressure]
    The \emph{pressure of \(\varphi\) on \(\mathcal{D}\) at scale \(\epsilon\)} is given by
    \[
    P\bigl(\mathcal{D}, \varphi, \epsilon; f\bigr)
    \;=\;
    \limsup_{n \to \infty}
    \frac{1}{n}
    \log \Lambda_n^{\mathrm{sep}}\bigl(\mathcal{D}, \varphi, \epsilon; f\bigr).
    \]
    Then the \emph{pressure of \(\varphi\) on \(\mathcal{D}\)} is defined as
    \[
    P(\mathcal{D}, \varphi; f)
    \;=\;
    \lim_{\epsilon \to 0}\,
    P\bigl(\mathcal{D}, \varphi, \epsilon; f\bigr).
    \]
\end{definition}

Given \(Z \subset X\), define
\[
P(Z, \varphi, \epsilon; f) \;:=\; P(Z \times \mathbb{N}, \varphi, \epsilon; f).
\]
Observe that \(P(Z, \varphi; f)\) denotes the usual upper capacity pressure \cite{Bowen1974SomeSW}, and we often write \(P(\varphi; f)\) in place of \(P(X, \varphi; f)\) for the pressure on the entire space.

When \(\varphi = 0\), our definition yields the entropy of \(\mathcal{D}\):
\[
h(\mathcal{D}, \epsilon; f) \;=\; h(\mathcal{D}, \epsilon)
\;:=\;
P(\mathcal{D}, 0, \epsilon)
\quad\text{and}\quad
h(\mathcal{D})
\;=\;
\lim_{\epsilon \to 0} h(\mathcal{D}, \epsilon).
\tag{2.1}
\]

\begin{definition}[Equilibrium State]
    Let \(\mathcal{M}(f)\) be the set of \(f\)-invariant Borel probability measures, and \(\mathcal{M}_e(f)\) the subset of ergodic measures in \(\mathcal{M}(f)\). By the variational principle for pressure \cite[Theorem 9.10]{MR0648108},
    \[
    P(\varphi; f)
    \;=\;
    \sup_{\mu \in \mathcal{M}(f)}
    \Bigl(
    h_\mu(f)
    \;+\;
    \int \varphi \,d\mu
    \Bigr)
    \;=\;
    \sup_{\mu \in \mathcal{M}_e(f)}
    \Bigl(
    h_\mu(f)
    \;+\;
    \int \varphi \,d\mu
    \Bigr).
    \]
    A measure achieving this supremum is called an \emph{equilibrium state}.
\end{definition}

\subsection{Obstructions to expansivity, specification, and regularity}
Bowen showed in \cite{Bowen1974SomeSW} that if \((X, f)\) is expansive and has the specification property, and if \(\varphi\) satisfies certain regularity conditions, then there is a unique equilibrium state. We recall definitions and results from \cite{MR3552538}, which demonstrate that non-uniform analogues of Bowen's hypotheses are sufficient to guarantee uniqueness.

\subsubsection{Expansivity}   
Given a continuous map \(f : X \to X\), the \emph{infinite Bowen ball} of size \(\epsilon > 0\) around a point \(x \in X\) is defined by
\[
\Gamma_\epsilon(x)
\;:=\;
\bigl\{\, y \in X : d(f^k x, f^k y) < \epsilon 
\text{ for all } k \in \mathbb{\N} \bigr\}.
\]
If there exists \(\epsilon > 0\) such that \(\Gamma_\epsilon(x) = \{x\}\) for every \(x \in X\), then we say that \((X, f)\) is \emph{positively expansive}. Throughout the rest of the paper we drop ``positively'' and simply write \emph{expansive}, since all maps under consideration are non-invertible.

\begin{definition}
For \(f : X \to X\), the set of \emph{non-expansive points} at scale \(\epsilon\) is defined by
\[
\mathrm{NE}(\epsilon)
\;=\;
\bigl\{\, x \in X : \Gamma_\epsilon(x) \neq \{x\} \bigr\}.
\]
An \(f\)-invariant measure \(\mu\) is called \emph{almost expansive} at scale \(\epsilon\) if \(\mu\bigl(\mathrm{NE}(\epsilon)\bigr) = 0\).

Given a potential \(\varphi\), the \emph{pressure of obstructions to expansivity} at scale \(\epsilon\) is
\begin{align*}
    P_{\mathrm{exp}^+}^{\perp}(\varphi,\epsilon)
    \;&=\;
    \sup_{\mu \in \mathcal{M}_e(f)}
    \Bigl\{
    h_\mu(f) + \int \varphi \, d\mu
    : \mu\bigl(\mathrm{NE}(\epsilon)\bigr) > 0
    \Bigr\}
\end{align*}

This quantity is \emph{non-increasing} in \(\epsilon\), so we define a scale-free version by
\[
P_{\mathrm{exp}^+}^{\perp}(\varphi)
\;=\;
\lim_{\epsilon \to 0}
P_{\mathrm{exp}^+}^{\perp}(\varphi,\epsilon).
\]
\end{definition}

In particular, the scale-free version
\(
P_{\exp^+}^{\perp}(\varphi) 
\)
satisfies
\[
P_{\exp^+}^{\perp}(\varphi) \le P_{\exp^+}^{\perp}(\varphi,\epsilon)
\]
for any fixed \(\epsilon > 0\).

Roughly speaking, \( P^\perp_{\exp^+}(f,\epsilon) \) can be understood as follows. Recall that positive expansivity is equivalent to the condition
\[
\bigcap_{n \ge 0} B_n(x,\epsilon) = \{x\},
\]
which holds for every \( x \in X \) when \(\epsilon\) is sufficiently small. We define \( P^\perp_{\exp^+}(f,\epsilon) \) to be the supremum of the pressures of all ergodic measures that assign positive weight to the set of points where this condition fails.

\subsubsection{Specification}

\begin{definition}[W–specification at scale $\epsilon$]
A collection of orbit segments 
\[
   \mathcal D \subset X\times\N
   \quad\bigl(\text{with } (x,n)\text{ representing the segment }
   x,\,f(x),\dots,f^{n-1}(x)\bigr)
\]
has the \emph{(W)-specification property at scale $\epsilon>0$} if there exists 
\(\tau=\tau(\epsilon)\in\N\) such that, for every finite set of segments
\[
  \{(x_j,n_j)\}_{j=1}^{k}\subset\mathcal D,
\]
one can find ``gluing times'' 
\(\tau_1,\dots,\tau_{k-1}\in\N\) with \(\tau_i\le \tau\) for every
\(i\) and a point \(y\in X\) satisfying
\[
   d_{n_j}\!\bigl(f^{t_{j}}(y),x_j\bigr)<\epsilon
   \quad\text{for each }1\le j\le k,
\]
where
\[
  t_1:=0,\qquad 
  t_j:=\sum_{i=1}^{j-1}\bigl(n_i+\tau_i\bigr)\quad (2\le j\le k).
\]
Here \(d_{n}(x,y)=\max_{0\le m<n}d\bigl(f^{m}x,f^{m}y\bigr)\)
is the Bowen metric of length \(n\).

We say that \(\mathcal{D}\) has the \emph{(W)-specification property} if the above condition holds for every \(\epsilon > 0\). We say that \((X, f)\) has the (W)-specification property if \(X \times \mathbb{N}\) does. 
\end{definition}

As in \cite{MR3552538} we will sometimes say that \(\mathcal{G}\) has (W)-specification, or simply specification. This version of the specification is weak in the sense that the ``transition times'' \(\tau_i\) are only assumed to be bounded by, rather than equal to, \(\tau\).

\begin{definition}
A triple \(\bigl(\mathcal{P}, \mathcal{G}, \mathcal{S}\bigr) \subset (X \times \mathbb{N})^3\) is called a \emph{decomposition} for \((X, f)\) if for every \(x \in X\) and \(n \in \mathbb{N}\), there exist \(p, g, s \in \mathbb{N}\) with \(p + g + s = n\) such that
\[
(x, p) \;\in\; \mathcal{P},
\quad
\bigl(f^p x, g\bigr) \;\in\; \mathcal{G},
\quad
\bigl(f^{p+g} x, s\bigr) \;\in\; \mathcal{S}.
\]
In other words, every orbit segment can be decomposed into a “prefix” from \(\mathcal{P}\), a “good” core from \(\mathcal{G}\), and a “suffix” from \(\mathcal{S}\).

\end{definition}

Note that the symbol $(x, 0)$ denotes the empty set, and the functions $p, g, s$ are permitted to take the value zero.

\begin{definition}
The \emph{pressure of obstructions to specification} at scale \(\epsilon\) is defined by
\[
P_{\mathrm{spec}}^{\perp}(\varphi, \epsilon) 
\;:=\; 
\inf \Bigl\{ 
  \overline{P}\bigl(\mathcal{P} \cup \mathcal{S},\varphi, 3\epsilon\bigr)
\Bigr\}.
\]
where the infimum is taken over $\mathcal{P}, \mathcal{S} \subset X \times \N$ such that there exists a decomposition $(\mathcal{P},\mathcal{G},\mathcal{S}$) for $(X,f)$ for which every $\mathcal{G}$ has specification at scale $\epsilon$. We then define the scale-free version:
\[
P_{\mathrm{spec}}^{\perp}(\varphi)
\;=\;
\lim_{\epsilon \to 0}
P_{\mathrm{spec}}^{\perp}(\varphi, \epsilon).
\]
\end{definition}
Again, since \(P_{\mathrm{spec}}^{\perp}(\varphi, \epsilon)\) is monotonic as a function of \(\epsilon\), this limit exists. In contrast to \(P_{\mathrm{exp}}^{\perp}(\varphi, \epsilon)\), \(P_{\mathrm{spec}}^{\perp}(\varphi, \epsilon)\) is non-decreasing in \(\epsilon\).

\subsubsection{Regularity for potential}
\begin{definition}[Bowen Property]
Given \(\mathcal{D} \subset X \times \mathbb{N}\), we say a function \(\varphi : X \to \mathbb{R}\) has the \emph{Bowen property} on \(\mathcal{D}\) at scale \(\epsilon\) if there is a constant \(K = K(\varphi, \mathcal{D}, \epsilon)\) such that
\[
\bigl| S_n \varphi(x) - S_n \varphi(y) \bigr| < K
\]
for every \((x, n) \in \mathcal{D}\) and every \(y \in B_n(x, \epsilon)\). A function \(\varphi\) has the \emph{Bowen property} on \(\mathcal{D}\) if it satisfies the Bowen property on \(\mathcal{D}\) at some scale \(\epsilon\) (and thus at all smaller scales as well).
\end{definition}

\subsection{General results on uniqueness of equilibrium states}
Our main tool for existence and uniqueness of equilibrium state is \cite[Theorem 5.5]{MR3552538}. 

\begin{theorem}[{\cite[Theorem 5.5 and Theorem 5.7]{MR3552538}}]\label{ct}
Let \(X\) be a compact metric space and \(f \colon X \to X\) a continuous map. Let 
\(\varphi \colon X \to \mathbb{R}\) be a continuous potential function. Suppose that 
\[
P_{\mathrm{exp}^+}^{\perp}(\varphi) \;<\; P_{\mathrm{top}}(\varphi),
\]
and that \((X, f)\) admits a decomposition \(\bigl(\mathcal{P}, \mathcal{G}, \mathcal{S}\bigr)\) with the following properties:
\begin{enumerate}
    \item \(\mathcal{G}\) has (W)-specification at any scale;
    \item \(\varphi\) has the Bowen property on \(\mathcal{G}\);
    \item \(P_{\mathrm{spec}}^{\perp}(\varphi) \;<\; P_{\mathrm{top}}(\varphi)\).
\end{enumerate}
Then there is a unique equilibrium state \(\mu\) for \(\varphi\). Moreover, \(\mu\) satisfies the following upper large deviations bound:
if \(A\) is any weak*-closed and convex subset of the space of Borel probability
measures on \(X\), then
\[
\limsup_{n\rightarrow\infty} \frac{1}{n} \log\mu\left(\mathcal{E}_n^{-1}(A)\right) \le \sup_{\nu\in A\cap \mathcal{M}(f)}\left(h_{\nu}(f)+ \int \varphi \,d\nu - P_{\mathrm{top}}(\varphi)\right),
\]
where \(\mathcal{E}_n(x) = \frac{1}{n}\sum_{k=0}^{n-1} \delta_{f^kx}\) is the \(n\)-th empirical measure associated to \(x\).
\end{theorem}

We note that the theorem we apply assumes that the pressures of obstructions to both expansivity and specification are defined in a scale‐free manner.

\section{Viana maps and their properties}\label{sec:viana-maps}
We begin by gathering the necessary materials for the construction of Viana maps with the desired decomposition properties.

We consider from here on these maps \(f\) close to \(f_{\alpha}\) restricted to \(S^1 \times I\). The main properties of \(f\) in a \(C^3\) neighborhood of \(f_{\alpha}\) are summarized below; see \cite{MR1471866}, \cite{MR1743717}.
\begin{itemize}
    \item \textit{Non-uniform expanding}, that is, for $\alpha$ sufficiently small, there exists a positive constant $c_0 > 0$ such that any map $f$ sufficiently close to $f_{\alpha}$ in the $C^3$ topology has both Lyapounov exponents greater than $c_0$ at Lebesgue almost every point, that is, for Lebesgue a.e. $x=(\theta, t) \in S^1 \times I$.
    \[\liminf_{n\rightarrow \pm \infty} \frac{1}{n} \log \|Df^n(x)\| > c_0;\]

    \item \textit{Slow approximation to the critical set}, that is, for every \(\gamma>0\) there exists \(\delta > 0\) such that for Lebesgue a.e. $x \in S^1 \times I$, the following holds
    \[\limsup_{n\rightarrow\infty} \frac{1}{n}\sum_{j=0}^{n-1}-\log \, [dist_{\delta}(f^j(x),\mathcal{C})]\leq \gamma\] 
    Here \(dist_{\delta}\) is the \(\delta\)-truncated distance defined by
    \[dist_{\delta}(x,\mathcal{C}) = 
    \begin{cases}
         1, &\quad \text{if \(dist(x,\mathcal{C}) \geq \delta\)},\\
         dist(x,\mathcal{C}), &\quad \text{otherwise}
    \end{cases};\]

    \item Existence and uniqueness of an ergodic absolutely continuous invariant (SRB) measure \(\mu_{\mathrm{SRB}}\);

    \item The attractor of \(f\) within the forward-invariant region \(S^1 \times I\) is defined as the intersection
    \[
    \Lambda =\bigcap_{n\geq 0} f^n(S^1\times I)
    \]
    of all forward images of \(S^1 \times I\), which is just \(\Lambda = f^2(S^1 \times I)\), as long as the interval \(I\) is properly chosen; 

    \item 
    Topological mixing: for any open set \( O \subset S^1 \times I \), there exists \(n= n(O) \in \mathbb{N} \) such that 
    \[
    f^{n}(O) = \Lambda.
    \]  
\end{itemize}

Next, we study the dynamics of the tangent bundle for Viana maps. Denote by \( f_0 \) the direct product map 
\[
    f_0 = g \times h: S^1 \times I \to S^1 \times I, \quad (\theta, t) \mapsto (d\theta \mod 1, a_0 - t^2).
\]
Since \( Df_0 \cdot \frac{\partial}{\partial \theta} = d \cdot \frac{\partial}{\partial \theta} \) and \( Df_0 \cdot \frac{\partial}{\partial t} = -2t \cdot \frac{\partial}{\partial t} \), with \( d \geq 16 > 4 \) and \( I \subset (-2, 2) \), it is clear that \( f_0 \) is a partially hyperbolic map of the form \( T_{(\theta,t)}(S^1 \times I) = E^u \oplus E^c \), where \( E^u \) is uniformly expanding and \( E^c \) is dominated by \( E^u \). From the theory of partial hyperbolicity, we know that if \( \alpha \) is sufficiently small, the map \( f_{\alpha} \) (like any \( C^1 \)-nearby map) also admits a partially hyperbolic splitting of the form \( E^u \oplus E^c \), which varies continuously. As a consequence, we get

\begin{proposition}\label{unstable_Lyap_exp}
    Given $\epsilon > 0$, if $\alpha$ is sufficiently small, the Lyapounov exponent
    $$\lambda^u(x, f):= \lim_{n\rightarrow\infty}\frac{1}{n}\log \|Df^n|_{E^u(x)}\|$$
    associated to $E^u$ at every point $x$ satisfies
    $$\log(d-\epsilon)\leq \lambda^u(x, f) \leq \log(d+\epsilon),$$
    for any $f$ close to $f_{\alpha}$.
\end{proposition}

By applying the previous theorems we derive a lower bound on the entropy of the SRB measure \( \mu_{\mathrm{SRB}} \). By Pesin's formula, the entropy \( h_{\mu_{\mathrm{SRB}}}(f) \) of \( \mu_{\mathrm{SRB}} \) is given by the integrated sum of the positive Lyapunov exponents of \( \mu_{\mathrm{SRB}} \). Therefore, by \cref{unstable_Lyap_exp}, we obtain the following bound:
\[
h_{\mu_{\mathrm{SRB}}}(f) \geq \log(d - \epsilon) + c_0.
\]

\section{Hyperbolic Times and Non-uniformly Expanding maps}\label{sec:hyperbolic-times}
To advance the proof, we first recall the notion of hyperbolic time for maps with critical points. The setup is as follows: Let \( f: M \to M \) be a \( C^2 \) map that is a local diffeomorphism except on a set \( \mathcal{C} \subset M \) of zero Lebesgue measure. Assume that \( f \) behaves like a power of the distance to the critical set \( \mathcal{C} \). Specifically, there exist constants \( B > 1 \) and \( l > 0 \) such that, for all \(x, y \notin \mathcal{C} \) with \( 2 \, \text{dist}(x, y) < \text{dist}(x, \mathcal{C}) \), and for any \( v \in T_xM \), we have:
   \begin{itemize}
        \item $\frac{1}{B} dist(x, \mathcal{C})^l \leq \frac{\|Df(x)v\|}{\|v\|} \leq B dist(x, \mathcal{C})^l;$
        \item $| \log \|Df (x)^{-1}\| - \log \|Df (y)^{-1}\|| \leq B \frac{dist(x,y)}{dist(x, \mathcal{C})^l} ;$
        \item $| \log |\det Df (x)^{-1}| - \log |\det Df (y)^{-1}|| \leq B \frac{dist(x,y)}{dist(x, \mathcal{C})^l} ;$
   \end{itemize}

As the reader can easily verify, the Viana maps satisfy all the previous assumptions, since these conditions are $C^2$-open and $f_{\alpha}$ satisfies them.

For the definition of hyperbolic times, we fix $0 < b < \min\{1/2, 1/(2\beta)\}$. 
\begin{definition}\label{hyp_times}
Given $0 < \sigma < 1$ and $\delta > 0$, we say that $n$ is a $(\sigma,\delta)$-hyperbolic time for $x$ if, for all $1 \leq k \leq n$,
$$\prod_{j=n-k}^{n-1}\|Df(f^j(x))^{-1}\| \leq \sigma^k$$ 
and 
$$dist_{\delta}(f^{n-k}(x),\mathcal{C}) \geq \sigma^{bk}$$
\end{definition}

Hyperbolic times are useful because they have the following key property:
\begin{proposition}[{\cite[page 377]{MR1757000}}] \label{hyp_times_prop}
     Given $\sigma < 1$ and $\delta > 0$, there exists $\delta_1 > 0$ such that if $n$ is a $(\sigma,\delta)$-hyperbolic time of $x$, then there exists a neighborhood $V_x$ of $x$ such that:
    \begin{itemize}
        \item $f^n$ maps $V_x$ diffeomorphically onto the ball $B_{\delta_1}(f^n(x))$;
        \item For every $1\leq k \leq n$ and $y,z\in V_x$, $$dist(f^{n-k}(y), f^{n-k}(z)) \leq \sigma^{k/2} dist(f^n(y), f^n(z))$$.
    \end{itemize}
\end{proposition}

\begin{definition}
    For a fixed $\sigma<1$, we define the expanding set $H(\sigma)$ to be the set of points $x\in M$ with the following two properties:
$$\limsup_{n\rightarrow \infty} \frac{1}{n}\sum_{j=0}^{n-1}\log(\|Df(f^j(x)^{-1})\| \leq 3\log\sigma < 0$$
and, for any $\gamma > 0$ there is some $\delta > 0$ satisfying
$$\limsup_{n\rightarrow \infty} \frac{1}{n}\sum_{j=0}^{n-1}-\log(dist_{\delta}(f^j(x),\mathcal{C})) \leq \gamma$$
\end{definition}

\begin{proposition}[{\cite[page 379]{MR1757000}}]\label{infty_hyp_time} 
    Given $\sigma < 1$, there exist $\nu > 0$ and $\delta > 0$ depending only on $\sigma$ and $f$ such that, given any $x \in H(\sigma)$ and $N \geq 1$ sufficiently large, there are $(\sigma,\delta)$-hyperbolic times $1\leq n_1 <\cdots <n_l \leq N$ for $x$ with $l\geq \nu N$.
\end{proposition}

\section{Non-uniform expansivity}\label{sec:exp}
In this section, we show that Viana maps are systems for which expansivity fails, but this failure occurs in a lower entropy setting, giving rise to a form of non-uniform expansivity. 

\subsection{Non-expansive points}
We first explain how Viana maps fail to be expansive. A key source of non-expansivity in Viana maps is the repeated folding of points on the fiber direction. More precisely, one constructs a small vertical segment crossing the critical set \(\{t=0\}\) and observes that its forward image is folded by the quadratic term 
\(t\mapsto a_0-t^2\). Although iterates away from the critical region could in principle separate points, the skew-product structure makes it possible that for some choice \(\theta\) this same vertical segment (or a small sub-segment of it) is returned to the critical region sufficiently often, so that before substantial divergence occurs, another iterate causes it to fold again. This process produces pairs of distinct points whose forward orbits remain arbitrarily close for all times. In particular, for any arbitrarily small \(\epsilon\), there always exist distinct points whose orbits remain \(\epsilon\)-close for every iterate, so the infinite Bowen ball of size \(\epsilon\) never degenerates to a single point. Consequently, Viana maps fail to be expansive at every scale.

\subsection{Slow recurrence}
To overcome the loss of expansion caused by folding, we need to control the recurrence to the critical set $\mathcal{C}$ of generic points of invariant ergodic measures whose Lyapounov exponents are regular.
\begin{lemma}\label{dist_integrable}
    Let $\eta$ be an ergodic measure such that $\lambda^c(\eta) > -\infty$, where $\lambda^c$ is the Lyapounov exponent of $\eta$ associated to $E^c$. Then,
    $$\int |\log \, dist (x, \mathcal{C})| \,d\eta < \infty.$$
\end{lemma}
\begin{proof}
    Since $\eta$ is ergodic, $\lambda^c(\eta) =  \int \log\|Df|_{E^c}\|\,d\eta > -\infty$. On the other hand, the definition of $f_{\alpha}$ and $\alpha$ sufficiently small implies $\frac{1}{3} \|Df(x)|_{E^c}\| \leq dist (x, \mathcal{C}) \leq 3 \|Df(x)|_{E^c}\|$. These two facts together finish the proof.
\end{proof}

\begin{lemma}\label{slow_rec}
    Let \(\eta\) be an ergodic measure with \(\lambda^c(\eta) > -\infty\). Then, for any \(\gamma > 0\), there exists \(\delta > 0\) such that for \(\eta\)-almost every point \(x\),
    \[
    \limsup_{n \to \infty}
    \frac{1}{n} \sum_{j=0}^{n-1} 
    \bigl(-\!\log \, dist_\delta\bigl(f^j(x), \mathcal{C}\bigr)\bigr)
    \;\le\;
    \gamma.
    \]
\end{lemma}

\begin{proof}
    By \cref{dist_integrable}, the function \(\bigl|\log \, dist\bigl(x, \mathcal{C}\bigr)\bigr|\) is \(\eta\)-integrable, and in particular, \(\eta(\mathcal{C}) = 0\). Using these facts together with the definition of the \(\delta\)-truncated distance \(dist_\delta\), we deduce that for any \(\gamma > 0\), there exists \(\delta > 0\) such that
    \[
    \int 
    -\log \, dist_\delta\bigl(x,\mathcal{C}\bigr)
    \, d\eta 
    \;\le\; 
    \gamma.
    \]
    Because \(\eta\) is ergodic, an application of Birkhoff’s Ergodic Theorem to the function \(-\log \, dist_\delta\bigl(x, \mathcal{C}\bigr)\) completes the proof.
\end{proof}

\subsection{Almost expansivity}
Next we prove that, for Viana maps, any ergodic \(f\)-invariant measure with large pressure must be positively expansive. We prove this using the strategy outlined below; see \cite{MR2032493} and \cite{arbieto2005maximalentropymeasuresviana}.
\begin{itemize}
    \item Select a certain set \(\mathcal{K}\) of ergodic \(f\)-invariant measures, which are natural candidates for attaining the variational principle.
    \item Prove that any measure outside \(\mathcal{K}\) have strictly smaller pressure than the topological pressure hence cannot be an equilibrium state.
    \item Show that any measure inside \(\mathcal{K}\) is an expanding measure, and any expanding measure is almost positively expansive.
\end{itemize}
We begin by describing \(\mathcal{K}\), the set of candidate measures for achieving the variational principle. Define
\[
\mathcal{K}
\;=\;
\Bigl\{
  \mu \in \mathcal{M}_e(f) 
  : 
  \lambda^c(\mu) 
  \;>\;
  \tfrac{c_0}{2}
\Bigr\}.
\]
That is, \(\mathcal{K}\) consists of those ergodic \(f\)-invariant measures whose central Lyapunov exponent is strictly bounded below by \(\tfrac{c_0}{2}\). 
The set \(\mathcal{K}\) is nonempty, as it contains the unique SRB measure \(\mu_{\mathrm{SRB}}\), which has \(\lambda^c(\mu_{\mathrm{SRB}}) > c_0\) by Pesin's entropy formula.

\subsubsection{Low Pressure Outside \(\mathcal{K}\)}
The measures in \(\mathcal{K}\) are natural candidates for realizing the variational principle because the measures outside of \(\mathcal{K}\) have small pressure:

\begin{lemma}\label{low_pressure_outside}
For every measure \(\eta \notin \mathcal{K}\), we have
\[
P_{\eta}(\varphi) < \sup_{\mu \in \mathcal{M}_e(f)} P_{\mu}(\varphi) = P_{\mathrm{top}}(\varphi).
\]
\end{lemma}

\begin{proof}
We compare \(P_{\eta}(\varphi)\) with \(P_{\mu_{\mathrm{SRB}}}(\varphi)\). Since \(\eta \notin \mathcal{K}\), it follows that \(\lambda^c(\eta) \le \frac{c_0}{2}\). Therefore,
\begin{align*}
    P_{\eta}(\varphi) = h_{\eta}(f) + \int \varphi \, d\eta 
    &\le \log(d+\epsilon) +\frac{c_0}{2}+ \int \varphi \, d\eta \\
    &< \log(d+\epsilon) + \frac{c_0}{2} +\left(\frac{c_0}{2}-\log \frac{d+\epsilon}{d-\epsilon}\right)+ \int \varphi \, d\mu_{\mathrm{SRB}}  \\
    &\le h_{\mu_{\mathrm{SRB}}}(f) + \int \varphi \, d\mu_{\mathrm{SRB}} = P_{\mu_{\mathrm{SRB}}}(\varphi) .
\end{align*}

Here, the first inequality follow from Ruelle's inequality, the second from our assumption that \(\sup \varphi - \inf \varphi < \frac{c_0}{2}-\log \frac{d+\epsilon}{d-\epsilon}\) which implies \(\int \varphi \,d\eta - \int \varphi \,d \mu_{\mathrm{SRB}} < c_0/2\), and the third from Pesin's entropy formula. By the variational principle, this proves that \(P_{\eta}(\varphi)<P_{\mathrm{top}}(\varphi)\).
\end{proof}

\subsubsection{Expanding Measures}
We define the notion of expanding measures originally introduced in \cite{MR1757000}. 

\begin{definition}
     A measure \(\mu\) is called a \(\sigma^{-1}\)-expanding measure if \(x \in H(\sigma)\) for \(\mu\)-almost every point \(x\). In other words, \(\mu(H(\sigma)) = 1\). Here \(H(\sigma)\) is the \(\sigma^{-1}\)-expanding set defined in \cref{hyp_times}.
\end{definition}

We now show that there exists a single \(\sigma > 1\) such that every \(\mu \in \mathcal{K}\) is \(\sigma^{-1}\)-expanding. As will become clear later, a uniform choice of \(\sigma\) across all \(\mu \in \mathcal{K}\) is necessary for establishing that every ergodic measure in \(\mathcal{K}\) is almost expansive at the \emph{same} scale.

\begin{lemma}
    Any \(\mu \in \mathcal{K}\) is \(\sigma^{-1}\)-expanding with \(\sigma = \exp\left(-\frac{c_0}{6}\right)\).
\end{lemma}

\begin{proof}
    Since \(\mu\) is ergodic, we have
    \[
    \lambda^c(x) = \lambda^c(\mu) = \int \log\|Df(x)|_{E^c}\| \, d\mu > \frac{c_0}{2}
    \]
    for \(\mu\)-almost every \(x\). By Birkhoff's ergodic theorem, this is equivalent to
    \[
    \lim_{n \to \infty} \frac{1}{n} \sum_{j=0}^{n-1} \log \|Df(f^j(x))^{-1}\| \leq 3 \cdot \left( -\frac{c_0}{6} \right) < 0.
    \]
    because the central direction \(E^c\) is one-dimensional.
    
    This fact, combined with \cref{slow_rec}, gives us the assumptions needed for \(H(\sigma)\) at \(\mu\)-almost every point.
\end{proof}

\subsubsection{Expansivity inside \(\mathcal{K}\)}
Finally, we conclude the proof that any \(\sigma^{-1}\)-expanding measure \(\mu\) is almost expansive at scale \(\delta_1\) which implies every \(\mu\in\mathcal{K}\) is almost expansive.

\begin{lemma}\label{expansivity}
    For any \(\sigma^{-1}\)-expanding measure \(\mu\), there exists \(\delta_1 > 0\) such that for any \(\epsilon < \delta_1\) and for \(\mu\)-almost every \(x\),
    \[
    \Gamma^+_x(\epsilon) = \{x\}.
    \]
\end{lemma}

\begin{proof}
    \cref{infty_hyp_time} guarantees that for \(\mu\)-almost every \(x\), there are infinitely many hyperbolic times \(n_i(x)\), since \(\mu \in \mathcal{K}\). Therefore, applying \cref{hyp_times_prop}, we conclude that there exists some \(\delta_1 > 0\) such that if \(y \in \Gamma^+_x(\epsilon)\) with \(\epsilon < \delta_1\), then for any \(n_i\), we have
    \[
    \text{dist}(x, y) \leq \sigma^{n_i/2} \, \text{dist}(f^{n_i}(x), f^{n_i}(y)) \leq \sigma^{n_i/2} \epsilon.
    \]
    Since \(n_i \to \infty\) as \(i \to \infty\), we deduce that \(x = y\).
\end{proof}

\subsection{Pressure gap}
\begin{theorem}\label{expansivity-gap}
We have the pressure gap 
\[
P_{\mathrm{exp}}^{\perp}(\varphi) \;<\; P_{\mathrm{top}}(\varphi).
\]
\end{theorem}

\begin{proof}
By definition, for any \(\epsilon < \delta_1\),
\[
P_{\mathrm{exp}}^{\perp}(\varphi,\epsilon)
\;=\;
\sup_{\mu \in \mathcal{M}_e(f)}
\Bigl\{
h_\mu(f) \;+\; \int \varphi \, d\mu
\;:\;
\mu\bigl(\mathrm{NE}(\epsilon)\bigr) > 0
\Bigr\}.
\]
Consider two cases for a measure \(\mu\):

\smallskip
\noindent
\textbf{Case 1:} \(\mu \notin \mathcal{K}\). 
By \cref{low_pressure_outside}, 
\[
P_{\mu}(\varphi) 
\;<\; 
P_{\mathrm{top}}(\varphi).
\]
Hence, such measures cannot achieve a pressure at \(P_{\mathrm{top}}(\varphi)\).

\smallskip
\noindent
\textbf{Case 2:} \(\mu \in \mathcal{K}\).
By \cref{expansivity}, there exists \(\delta_1 > 0\) so that for all \(\epsilon < \delta_1\),
\[
\mu\bigl(\mathrm{NE}(\epsilon)\bigr) 
\;=\; 
0.
\]
Thus, for sufficiently small \(\epsilon\), these measures place no mass on \(\mathrm{NE}(\epsilon)\) and do not contribute to \(P_{\mathrm{exp}}^{\perp}(\varphi,\epsilon)\).

In summary, for small enough \(\epsilon\), every measure that contributes to \(P_{\mathrm{exp}}^{\perp}(\varphi,\epsilon)\) lies outside \(\mathcal{K}\), and each such measure has pressure below \(P_{\mathrm{top}}(\varphi)\). Consequently,
\[
P_{\mathrm{exp}}^{\perp}(\varphi)
\;=\;
\lim_{\epsilon\rightarrow 0} P_{\mathrm{exp}}^{\perp}(\varphi,\epsilon) 
\;<\; 
P_{\mathrm{top}}(\varphi).
\]
This completes the proof.
\end{proof}

\section{Non-uniform specification and regularity}\label{sec:spec}
In this section, we first give an explicit description of how to decompose the collection of all orbit segments \((x,n)\) into \((\mathcal{P}, \mathcal{G},\mathcal{S})\). We show that \(\mathcal{G}\) has the specification property, and that orbits not in \(\mathcal{G}\) carry small pressure. In addition, Any Hölder continuous \(\varphi\) has the Bowen property on \(\mathcal{G}\) at any scale \(\epsilon\).

\subsection{Construction of the orbit decomposition}
Let \(\psi^c(x) = \log \|Df|_{E^c(x)}\|\), so that \(\frac{1}{n} S_n \psi^c(x)\) represents the average expansion (or contraction) undergone by an orbit segment \((x, n)\) in the center direction. Additionally, given \(\mu\in \mathcal{M}_e(f)\), \(\lambda^c(\mu) =\int \psi^c\, d\mu\) is the center Lyapunov exponent of \(\mu\).

To define the orbit decomposition (\cref{fig}), for some small \(r>0\) we consider the following collection of orbit segments: 
\begin{equation}\label{orb-decomp}
\begin{split}
    \mathcal{P} &:=\emptyset, \\
    \mathcal{G} &:=\{(x,n)\in \Lambda\times \N:\,S_{k}\psi^c(f^{n-k}x)\geq kr,\  \forall\, 0< k\leq n\},\\
    \mathcal{S} &:=\{(x,n)\in \Lambda\times \N:\, S_n\psi^c(x)< nr\}.
\end{split}   
\end{equation}

The collection \(\mathcal{G}\) is chosen so that along orbit segments from \(\mathcal{G}\), the center-unstable manifolds are uniformly backward contracting. In addition, \(f\) is locally invertible along these ``good'' orbits. 

Together with the trivial collection \(\{(x,0) : x \in \Lambda\}\) for \(\mathcal{P}\), these collections define a decomposition of any orbit \((x, n) \in \Lambda \times \mathbb{N}\) as follows: let \(s\) be the largest integer in \(\{0, \cdots, n\}\) such that \( S_{s}\psi^c(f^{n-s}x) < sr\), meaning \((f^{n-s}x, s) \in \mathcal{S}\).

We must then have \((x,  n-s) \in \mathcal{G}\), since if \(S_{k_0}\psi^c(f^{n-s-k_0} x) < k_0r\) for some \(0 \leq k_0 \leq n-s\), then 
\[
S_{s+k_0}\psi^c(f^{n-s-k_0}x) = S_{k_0}\psi^c(f^{n-s-k_0}x) + S_{s}\psi^c(f^{n-s}x) < (s+k_0)r,
\]
which would contradict the maximality of \(s\).

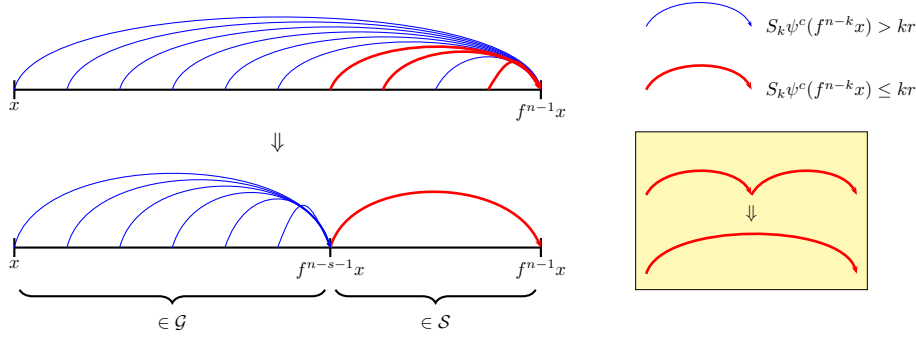
\begin{figure}[ht]
  \centering
  \begin{tikzpicture}[thick,scale=0.7, transform shape]  
\coordinate (x0) at (0,0);
\coordinate (xn) at (10,0);
\draw (x0) -- (xn);

\foreach \p/\lab in {x0/$x$,xn/$f^{n-1}x$}{
  \draw (\p) ++(0,-.18) -- ++(0,.36);
  \node[below=3pt] at (\p) {\lab};
}

\def\flat{0.7}  

\foreach \start [count=\k from 0] in {0,...,9}{%
  \pgfmathsetmacro{\h}{\flat*(1 + 0.18*(9-\k))}
  \draw[blue,->,line width=.4pt]
        (\start,0)
        .. controls +(.4,\h) and +(-.45,\h) ..
        (10,0);
}

\foreach \start in {9,7,6}{%
  \pgfmathsetmacro{\h}{\flat*(1 + 0.18*(9-\start))}  
  \draw[red,->,line width=1pt]              
        (\start,0)
        .. controls +(.4,\h) and +(-.45,\h) ..
        (10,0);
}

\node at ($(x0)!0.5!(xn)+(0,-1)$) {\Large$\Downarrow$};

\coordinate (x1) at (0,-3);
\coordinate (xs) at (6,-3);
\coordinate (xn) at (10,-3);
\draw (x1 |- xs) -- (xn);

\foreach \p/\lab in {x1/$x$,xs/$f^{n-s-1}x$,xn/$f^{n-1}x$}{
  \draw (\p) ++(0,-.18) -- ++(0,.36);
  \node[below=3pt] at (\p) {\lab};
}

\draw[red,->,line width=1pt]
      (xs |- xn) .. controls +(.4,1.4) and +(-.6,1.4) .. (xn);

\def\flat{0.9}  

\foreach \start in {0,1,2,3,4,5}{%
  \pgfmathsetmacro{\h}{\flat*(1+0.18*(6-\start)}    
  \draw[blue,->,line width=.4pt]
        (\start,-3) 
        .. controls +(.4,\h) and +(-.45,\h) ..
        (6,-3);
}

\draw[decorate,decoration={brace,mirror,amplitude=5pt}]
      ($(x1 |- xs)+(0.1,-0.8)$) -- ($(xs)+(-0.1,-0.8)$)
      node[midway,below=10pt] {$\in \mathcal{G}$};

\draw[decorate,decoration={brace,mirror,amplitude=5pt}]
      ($(xs)+(0.1,-0.8)$) -- ($(xn)+(-0.1,-0.8)$)
      node[midway,below=10pt] {$\in \mathcal{S}$};

\coordinate (legB0) at (12,1.2);
\coordinate (legB1) at (14,1.2);
\draw[blue,->, line width=0.4pt] (legB0) .. controls +(0.3,.6) and +(-0.3,.6) .. (legB1);
\node[right=4pt] at (legB1) {$S_{k}\psi^c(f^{n-k}x)> kr$};

\coordinate (legR0) at (12,0);
\coordinate (legR1) at (14,0);
\draw[red,->,line width=1pt] (legR0) .. controls +(0.3,.6) and +(-.3,.6) .. (legR1);
\node[right=4pt] at (legR1) {$S_{k}\psi^c(f^{n-k}x)\le kr$};

\begin{scope}[shift={(12,-2)}]
  \fill[yellow!35] (-0.2,-1.8) rectangle (4.2,1.2);
  \draw[black, line width=.2pt] (-0.2,-1.8) rectangle (4.2,1.2);

  \foreach \x/\h in {0/1.0,2/1.0}{
    \draw[red,->,line width=1pt]
          (\x,0) .. controls +(0.3,.6) and +(-0.3,.6) .. (\x+2,0);
  }

  \node at (2,-0.35) {$\Downarrow$};

  \draw[red,->,line width=1pt]
        (0,-1.5) .. controls +(0.3,1) and +(-0.3,1) .. (4,-1.5);
\end{scope}

\end{tikzpicture}
  \caption{Orbit decomposition}
  \label{fig}
\end{figure}

\subsection{Specification on \texorpdfstring{$\mathcal{G}$}{G}}
In the uniformly hyperbolic setting, a key step in establishing specification on a mixing, locally maximal hyperbolic set \(\Lambda\) is that for any \(\delta > 0\), there exists \(N \in \mathbb{N}\) such that for all \(x, y \in \Lambda\) and \(n \ge N\),
\[
f^n\bigl(W_\delta^u(x)\bigr) \;\cap\; W_\delta^s(y) \;\neq\; \varnothing.
\]
For Viana maps, we adapt this idea by replacing the local stable manifold with the point \(y\) itself, and the local unstable manifold with the center-unstable manifold \(W_\delta^{cu}(x)\). Indeed, Viana maps can be viewed as partially hyperbolic systems with one-dimensional unstable and center directions. Because each center-unstable manifold is merely a two-dimensional disk, and every such disk is eventually mapped onto \(\Lambda\) by the strong mixing property, the same construction effectively applies.

To prove specification on \(\mathcal{G}\), we begin by establishing that every ``good'' orbit in \(\mathcal{G}\) exhibits backward uniform contraction:
\begin{lemma}\label{back_contract}
    Fix \(\epsilon>0\). For every \((x,n)\in \mathcal{G}\) and \(y \in B_n(x,\epsilon)\), we have
    \[
    d\bigl(f^k x, f^k y\bigr) 
    \;\le\; 
    e^{-r(n-k)}\, d\bigl(f^n x, f^n y\bigr)
    \;\le\; 
    e^{-r(n-k)}\,\epsilon
    \quad
    \text{for all } 0 \le k \le n.
    \]
\end{lemma}

\begin{proof}
Since \(y \in B_n(x,\epsilon)\), the points \(f^j(x)\) and \(f^j(y)\) remain within \(\epsilon\)-distance for all \(0 \le j \le n\). In particular, 
\[
d\bigl(f^n x,\, f^n y\bigr) 
\;\le\;
\epsilon.
\]
To relate this distance at time \(n\) back to time \(k\), we observe that \(\|Df^{-(n-k)}(x)\| \le e^{-r(n-k)}\). It follows that
\[
d\bigl(f^k x,\, f^k y\bigr)
\;\le\;
e^{-r(n-k)}\, d\bigl(f^n x,\, f^n y\bigr)
\;\le\;
e^{-r(n-k)}\,\epsilon,
\]
as claimed.
\end{proof}

\begin{theorem}\label{specification}  
    The collection \(\mathcal{G}\) has specification at any scale \(\epsilon>0\).
\end{theorem}

\begin{proof}
We begin by observing that for any \(\epsilon >0 \), there exists some \(\tau = \tau(\epsilon) \in \mathbb{N}\) such that for any open disk \(B(x,\epsilon)\), we have \(f^j(B(x, \epsilon)) = \Lambda\) for some \(j \leq \tau\). This follows from the compactness and the topological mixing property of \(f\).

Now, consider any finite set of orbit segments \((x_1, n_1), \cdots, (x_k, n_k)\) from \(\mathcal G\). For any \(i = 1, \ldots, k-1\), we can select \(\tau_i \leq \tau(\epsilon)\) such that the image of \(B(f^{n_i} x_i, \epsilon)\) under \(f^{\tau_i}\) covers \(B_{n_{i+1}}(x_{i+1}, \epsilon)\).

Since \((x_i, n_i) \in \mathcal{G}\), by \cref{back_contract}, we have backward uniform contraction along this ``good'' orbit: for all \(0 \leq m \leq n_i\)
\[d\bigl(f^m x_i, f^m y\bigr) 
    \;\le\; 
    e^{-r(n_i-m)}\, d\bigl(f^{n_i} x_i, f^{n_i} y\bigr)
    \;\le\; 
    e^{-r(n_i-m)}\,\epsilon.\]
On a two-dimensional manifold this property is enough to guarantee that \(f^{n_{i}}\) maps \(B_{n_i}(x_i, \epsilon)\) diffeomorphically onto \(B(f^{n_{i}} x_i, \epsilon)\). As a consequence of these, if we define 
\[
t_1=0 \quad \text{and} \quad  t_j=\sum_{i=1}^{j-1} (n_i + \tau_i)  \quad \text{ for } j=2, \cdots, k
\]
and let \(U_j\) be such that 
\[
f^{t_{j}} U_j = B_{n_{j}}(x_{j}, \epsilon),
\]
then we have a nested sequence of sets:
\[
U_k \subset \cdots 
\subset U_{j+1} \subset U_j \subset \cdots \subset U_1 
\]
which means
\[
\bigcap_{j=1}^{k} U_j \neq \emptyset.
\]

To complete the proof of the theorem, choose any \(y \in \bigcap_{j=1}^{k} U_j\), then \(y\) satisfies 
\[
d_{n_j}\left(f^{t_{j}}(y),x_j\right)<\epsilon \quad \text{for every }1\le j\le k
\]
\end{proof}

\subsection{Pressure gap}
We prove that if \(0<r\leq \frac{c_0}{2}\), then for any bounded potential \(\varphi\), the ``bad orbits'' carry less topological pressure than the whole space. Recall we define \(\psi^c(x) = \log \|Df|_{E^c(x)}\|\) and
    \[
    \mathcal{S} \;:=\; 
    \bigl\{ (x, n) \in \Lambda \times \mathbb{N} : S_n \psi^c(x) \leq nr \bigr\}.
    \]
We have the following pressure gap estimate:
\begin{theorem}\label{specification-gap}
    \( 
    P(\mathcal{P} \cup \mathcal{S}, \varphi)< P_{\mathrm{top}}(\varphi).
    \)
\end{theorem}

\begin{proof}
    We start by taking \(\mathcal{P} = \emptyset\). To describe \(P(\mathcal{S},\varphi)\), recall we have proven in \cref{low_pressure_outside} that
    \[
    \sup_{\mu \in \mathcal{M}_f} \Bigl\{ h_{\mu}(f) +\int \varphi \,d\mu:   \lambda^c(\mu) \leq \frac{c_0}{2} \Bigr\}
    < P_{\mathrm{top}}(\varphi).
    \]
    So for \(0<r\leq \frac{c_0}{2}\),
    \begin{equation}\label{eq:upper-bound-entropy}
    \sup_{\mu \in \mathcal{M}(f)} \Bigl\{ h_{\mu}(f) +\int \varphi \,d\mu:   \lambda^c(\mu) \leq r \Bigr\}
    < P_{\mathrm{top}}(\varphi).
    \end{equation}
    
    Since we define \(\psi^c(x) = \log \| Df|_{E^c(x)} \|\), for an \(f\)-invariant measure \(\mu\), its center Lyapunov exponent satisfies 
    \(\lambda^c(\mu) = \int \psi^c \, d\mu\). 
    
    The left-hand side of \eqref{eq:upper-bound-entropy} can be related to \(P(\mathcal{S},\varphi)\): Let \(E_n \subset \mathcal{S}_n\) be any \((n,\epsilon)\)-separated set. Consider the Borel probability measure 
    
    \[
    \nu_n \;=\;\frac{1}{\sum_{x\in E_n}e^{S_n\varphi(x)}}\, \sum_{x\in E_n}\delta_x e^{S_n\varphi(x)}, \quad
    \mu_n \;=\; \frac{1}{n}\sum_{k=0}^{n-1} f_*^k \nu_n.
    \]
    
    Then half of the proof of the variational principle (see \cite[Theorem 8.6]{MR0648108}) shows that if \(\mu\) is a limit point of \(\mu_n\), then \(\mu\) is \(f\)-invariant and satisfies
    \[
    P_{\mu}(\varphi) \;=\; h_{\mu}(f) +\int \varphi\, d\mu\;\geq\; P(\mathcal{S}, \varphi, \epsilon).
    \]
    Moreover, by weak\(^*\)-convergence and the definition of \(\mathcal{S}\), we have
    \[
    \lambda^c(\mu) \;=\; \int \psi^c \, d\mu \;\leq\; r.
    \]
    Combining this with \eqref{eq:upper-bound-entropy}, we conclude that
    \[
    P(\mathcal{S}, \varphi) \;<\; P_{\mathrm{top}}(\varphi).
    \]
    Since \(\mathcal{P} = \emptyset\), this implies \(P(\mathcal{P} \cup \mathcal{S}, \varphi) < P_{\mathrm{top}}(\varphi)\). By definition, \(P_{\mathrm{spec}}^{\perp}(\varphi)<P_{\mathrm{top}}(\varphi)\), completing the proof.
\end{proof}

\subsection{Bowen property for potential}

\begin{theorem}\label{bowen-property}
Any H\"older continuous potential \(\varphi\) satisfies the Bowen property on \(\mathcal{G}\) at any scale \(\epsilon\).
\end{theorem}

\begin{proof}
Because \(\varphi\) is H\"older, there exist constants \(K>0\) and \(\alpha \in (0,1)\) such that
\[
\bigl|\varphi(x) - \varphi(y)\bigr|
\;\le\;
K\,\bigl(d(x,y)\bigr)^\alpha
\quad
\text{for all } x,y \in \Lambda.
\]
Take \((x,n)\in \mathcal{G}\) and \(y \in B_n(x,\epsilon)\). By \cref{back_contract},
\[
\bigl|S_n \varphi(x) \;-\; S_n \varphi(y)\bigr|
\;\le\;
K \sum_{k=0}^{n-1} 
\Bigl(d\bigl(f^k x,\, f^k y\bigr)\Bigr)^\alpha
\;\le\;
K \sum_{k=0}^{n-1} 
\Bigl(e^{-r\,(n-k)}\,\epsilon\Bigr)^\alpha.
\]
Re-indexing with \(j = n-k\) and noting the resulting geometric series,
\[
\sum_{k=0}^{n-1} 
\Bigl(e^{-r(n-k)}\,\epsilon\Bigr)^\alpha
\;\le\;
K\,\epsilon^\alpha 
\sum_{j=0}^{\infty} e^{-r\,\alpha\,j}
\;<\;
\infty.
\]
Hence, \(\bigl|S_n\varphi(x) - S_n\varphi(y)\bigr|\) is bounded uniformly in \(x,n,\) and \(y\), establishing the Bowen property at scale \(\epsilon\).
\end{proof}

\section{Proof of main theorems}\label{sec:proof-main}
We now complete the proof that if \(f\) is a \(C^3\)-perturbation of \(f_{\alpha}\) and \(\varphi:\Lambda \rightarrow \R\) which satisfy the hypotheses of \cref{main}, then the conditions of \cref{ct} are satisfied and hence there is a unique equilibrium state \(\mu\) for \((\Lambda, f, \varphi)\). Moreover, \(\mu\) satisfies the upper level-2 large deviation principle.

By \cref{expansivity-gap}, \(P_{\mathrm{exp}^+}^{\perp}(\varphi) \;<\; P_{\mathrm{top}}(\varphi)\). We define the decomposition \((\mathcal{P}, \mathcal{G}, \mathcal{S})\) as in \cref{orb-decomp}. In \cref{specification}, we showed that \(\mathcal{G}\) has specification at all scales. In \cref{bowen-property}, we showed that \(\varphi\) has the Bowen property on \(\mathcal{G}\). By \cref{specification-gap}, \(P\bigl(\mathcal{P} \cup \mathcal{S}, \varphi\bigr) \;<\; P_{\mathrm{top}}(\varphi)\). Thus, we see that under the hypotheses of \cref{main}, all the hypotheses of \cref{ct} are satisfied for the
decomposition \((\mathcal{P}, \mathcal{G}, \mathcal{S})\).

\printbibliography

\end{document}